\documentclass[final]{siamltex}
\usepackage{amsmath,amsopn,amssymb,epsfig,psfrag,showkeys} 
\usepackage{hyperref}

\newtheorem{example}{Example}[section]

\def\r2n{\mathbb{R}^{2n}}
\def\r2n2n{\mathbb{R}^{2n\times 2n}}
\def\rn{\mathbb{R}^{n}}
\def\rnn{\mathbb{R}^{n\times n}}

\newcommand{\bb}{\begin{bmatrix}}
\newcommand{\eb}{\end{bmatrix}}

\title{The Shifting Technique For Solving A Nonsymmetric Algebraic Riccati Equation
\footnote{Version~\today}}

\author{Chun-Yueh Chiang\thanks{Center for General Education,National Formosa
University,Huwei 632, Taiwan. {\tt (chiang@nfu.edu.tw)}}
      \and Matthew M.
Lin \thanks{Corresponding Author. Department of Mathematics, National Chung Cheng University, Chia-Yi 621, Taiwan. {\tt (mlin@math.ccu.edu.tw)}
This
research was supported in part by the National Science Council of
Taiwan under grant
99-2115-M-194-010-MY2.
}}

\begin{document}

\date{}
\maketitle

\begin{abstract}

%
%
%

This paper analyzes a special instance of nonsymmetric algebraic matrix Riccati equations arising from transport theory.
 Traditional approaches for finding the minimal nonnegative solution of the matrix Riccati equations are based on the fixed point iteration and the speed of the convergence is linear.  Relying on simultaneously matrix computation, a structure-preserving doubling algorithm (SDA) with quadratic convergence is designed for improving the speed of convergence. The difficulty is that the double algorithm with quadratic convergence cannot guarantee to work all the time. Our main trust in this work is to show that applied with a suitable shifted technique, the SDA is guaranteed to converge quadratically  with no breakdown.
Also, we modify the conventional simple iteration algorithm in the critical case to dramatically improve the speed of convergence. Numerical experiments strongly suggest that the total number of computational steps can be significantly reduced via the shifting procedure.

\end{abstract}

%



\pagestyle{myheadings} \thispagestyle{plain}

\section{Introduction}
The nonsymmeric algebraic Riccati equation (NARE), encountered in transport theory, is given by
\begin{equation}\label{NARE}
XCX-XD-AX+B=0,
\end{equation}
where $A,B,C$ and $D\in\mathbb{R}^{n\times n}$ are given by
\begin{equation}\label{Mpara}
A = \Delta-eq^\top, \quad
B = ee^\top,\quad
C = qq^\top, \quad
D = \Gamma-qe^\top.
\end{equation}
where
\begin{align}
\begin{array}{rclcl}
e & = &[1,\ldots,1]^\top\in\rn, && \\
q & =& [q_1,\ldots,q_n]^{\top}, &\mbox{with}& q_i=\frac{c_i}{2\omega_i},\\
\Delta&=&{\rm diag}([\delta_1,\ldots,\delta_n]), & \mbox{with}&
\delta_i=\frac{1}{c\omega_i(1+\alpha)},\\
\Gamma&=&{\rm diag}([d_1,\ldots,d_n]), &\mbox{with}&
d_i=\frac{1}{c\omega_i(1-\alpha)}.
\end{array}
\end{align}
The parameters, used to define the above matrices and vectors, satisfy
$0<c\leq 1$, $0\leq \alpha <1$ and the sequences are
 $0<\omega_n<\cdots<\omega_2<\omega_1<1$,
$c_i > 0$, $i=1,2,\ldots,n$, so that $\sum\limits_{i=1}^n c_i=1$.

For the physical meaning of the NARE~\eqref{NARE} and its corresponding parameters setup, the  reader is referred to~\cite{JuangLin1999}. Correspondingly, we define the corresponding  dual equation of \eqref{NARE}
\begin{equation}\label{DNARE}
YBY-YA-DY+C=0.
\end{equation}
To facilitate our discussion, we need
a nonsingular M-matrix or a singular irreducible M-matrix given by
\begin{equation}\label{Mmatrix}
M = \left [\begin{array}{cc}
D & -C\\
-B & A
\end{array}
\right]
\end{equation}
and its relative matrix
\begin{equation}\label{Hmatrix}
H =  JM,
\end{equation}
 where $J = \diag(I_n,-I_n)$ with $I_n$ to be the $n$ by $n$ identity matrix.
Our interest in this study is to find the minimal nonnegative solution $X$ of~\eqref{NARE}.
The existence conditions of the minimal nonnegative solution are shown by Juang et al. in~\cite{JuangLin1999}.
Iterative methods for solving this problem are numerous and can be divided into two major categories.

One is the method sharing a computational cost of $O(n^2)$ arithmetic operations (ops) per step, but converges linear or sublinear.  The representative method of the first category is the {\em simple iteration method} (SI) or {\em vector iteration method}, which is first proposed by Lu \cite{LuSIMAX05}.  This method is very simple and requires a computational cost of $4n^2$ ops per step. Recently, three more methods, modified simple iteration (MSI), nonlinear block Jacobi method (NBJ) and the nonlinear block Gass-Seidel method (NBGS), based on Lu's method are proposed in~\cite{Bai2008,Bao06}. It has been shown in~\cite{GuoLin10} that
if $(\alpha, c)\neq (0,1)$,
the speed of convergence of the NBGS is faster than the other three. Generally speaking, the iterative methods mentioned above can be classified as accelerated variants of the well-known fixed-point iterations. Also,  in~\cite{GuoLin10} we know that all these four methods can provide a linear convergence, if  $(\alpha, c)\neq (0,1)$ and a sublinear convergence, if $(\alpha, c) =  (0,1)$.

The other is a method with a cost of $O(n^3)$ ops but provides quadratic convergence.
Despite of the complexity, quadratically convergent methods are much to be desired in practice. There are several good algorithms that can  cause quadratic convergence, for example,  the Newton method~\cite{GuoLaub2000, Bini2008}
and the structure-preserving doubling algorithm (SDA)~\cite{GuoXULin2006, GuoBruMei2007}.  However, when $(\alpha, c) =  (0,1)$, both Newton method and the SDA algorithm are not always valid and require special attention.

In this work we fine-tune the customary SDA algorithm and make it always workable and quadratical convergent when solving~\eqref{NARE}.
The SDA algorithm was first proposed by Guo et.al.~\cite{ GuoXULin2006} for solving the NARE.
In~\cite{GuoXULin2006, Chiang2009}, it has been shown to be quadratically convergent, if $(\alpha, c)\neq (0,1)$ and linearly convergent with rate ${1}/{2}$, if $(\alpha, c) =  (0,1)$.
The later case is the so-called ``critical case" and is the most challenging problem that we will encounter when solving~\eqref{NARE}.
Roughly speaking,  the critical case embedded with some type of singularity, i.e., the matrix $H$ has two zero eigenvalues, that will significantly reduce the speed of convergence.
In~\cite{GuoBruMei2007},
Guo et al. propose an efficient method based on a single-shift technique  to accelerate the computation of the minimal nonnegative solution so that one singularity can be removed.
They also show that the doubling algorithm applied to the shifted equation of~\eqref{NARE} converges faster than the doubling algorithm applied to~\eqref{NARE}, if no breakdown occurs.
%
%
%
%
The approach of removing two zero eigenvalues of $H$ has also been introduced in~\cite{GuoBruMei2007}, but again
the convergence of  the doubling algorithm  cannot be guaranteed.
Our contribution in this paper, which we think is new in theory, is to provide a detailed analysis of changes in the eigenvalue distribution of matrices $H$ and $M$ as the shift procedures are employed. Through this discussion, the quadratic convergence of the SDA is guaranteed via the duble-shift technique to remove two singularities. Most important of all, the minimal nonnegative solution of the duble-shift model is shown to be equal to that of the original
model. We believe such results are the first detailed proofs of the eigenvalue
analysis of $H$ and $M$ and their corresponding matrices with shift procedures and should be of great significance for solving the NARE.

The organization of this paper is as follows. In Section 2, we review some of the main results and definitions that will be used for subsequent discussion. In Section 3, we provide a  complete discussion on the shifted modifications for the SDA algorithm.  We show that the SDA algorithm applied to the double-shift problem is always accessible and the solution obtained  from the double-shift problem is equal to the original NARE problem.
In Section 4, advantages of the shifting technique applied to the SI algorithm have been thoroughly investigated.
In Section 5, we present a few numerical experiments to show the practicability and effectiveness of the shifting procedure and concluding remarks are given in Section 6.


\section{Preliminaries} \label{sec_background}
In this section we briefly review the definitions of Z-matrix and M-matrix
and discuss further some of their properties which are required in the statements and in the proofs discussed in the following sections. We also
summarize the popular algorithm,  SDA, for our numerical experiments as we shall see below.

\subsection{Definition and Theorems}
In order to formalize our discussion, we start by introducing the following two definitions.

\begin{definition}
A matrix $A = (a_{ij})\in\mathbb{R}^{n\times n}$ is called a \emph{Z-matrix} if $a_{ij}\leq 0$ for all $i\neq j$.
\end{definition}
Note that for any Z-matrix $A$, there exists a matrix $B\in\mathbb{R}^{n\times n}$
with $B\geq 0$ and some $\alpha\in \mathbb{R}$ such that $A = \alpha I - B$ where $I$ is the identity matrix.
Also,  the definition of Z-matrix plays an important role in defining a given matrix to be an \emph{M-matrix}.
\begin{definition}
 A Z-matrix $A$ is called an \emph{M-matrix} if $A = \alpha I - B$ with $B\geq 0$ and $\alpha\geq\rho(B)$,
 where  $\rho(B)$ is the spectral radius of $B$. It is called a singular M-matrix if $\alpha = \rho(B)$ and a nonsingular M-matrix if $\alpha > \rho(B)$.
\end{definition}

There are a great many different conditions, which are mathematically intriguing and important for applications, that discuss the necessary and sufficient conditions
for a given Z-matrix to be an M-matrix. For our subsequent discussions, we apply
the following two well known and useful results in the study of M-matrices.
\begin{theorem}\label{thm_mmatrix1}\cite{Berman1994}
If $A \in\mathbb{R}^{n\times n}$ is a Z-matrix, the following statements are equivalent:
\begin{enumerate}
\item $A$ is a nonsingular M-matrix.
\item $\sigma(A)\subset \mathbb{C}_+$.
\item $Av > 0$ holds for some positive vector $v\in \mathbb{R}^n$.
\item $A^{-1} \geq 0$.
\end{enumerate}
\end{theorem}

\begin{theorem}\label{thm:prop1}\cite{GuoLaub2000}
If the matrix~\eqref{Mmatrix} is a nonsingular M-matrix, then the NARE~\eqref{NARE} and its dual equation~\eqref{DNARE} have minimal nonnegative solutions $X$ and $Y$, respectively. Also, matrices $D-CX$ and $A-BY$ are nonsingular M-matrix.
\end{theorem}

Note that the conditions we list here are only a selection from many more useful ones. See
~\cite{Berman1994, GuoLaub2000, Horn1991, Varga2000b} for a longer list
of conditions and references to the proofs.


\subsection{SDA Algorithm}
In~\cite{GuoXULin2006}, Guo et al. come up with the SDA algorithm for solving NARE problems and show that if the matrix $M$~\eqref{Mmatrix} is a nonsingular M-matrix (irreducible singular M-matrix~\cite{Chiang2009}), the SDA algorithm is well-defined and quadratically convergent (at least linearly convergent with rate
$1/2$). Its idea is based on the doubling transformation. For more details of the doubling transformation, the reader is referred to~\cite[Theorem 2.1]{GuoXULin2006}.
The SDA algorithm starts by choosing a positive scalar $\gamma$ with
\begin{equation*}
\gamma \geq  \max\left\{\max\limits_{1\leq i \leq n}a_{ii},\max\limits_{1\leq i \leq n}d_{ii}\right\}.
\end{equation*}
Let
\begin{subequations}
\begin{eqnarray*}
E_0 &=& I_n - 2\gamma V_\gamma^{-1}, \quad F_0 = I_n - 2\gamma W_\gamma^{-1}, \\
G_0 &=&  2\gamma D_\gamma^{-1}CW_\gamma^{-1}, \quad H_0 =2\gamma W_\gamma^{-1}BD_\gamma^{-1},
\end{eqnarray*}
\end{subequations}
where
\begin{subequations}
\begin{eqnarray*}
A_\gamma &=& A + \gamma I_n, \quad D_\gamma =  D + \gamma I_n,\\
W_\gamma &=& A_\gamma -BD^{-1}_\gamma C, \quad V_\gamma =  D_\gamma  - CA^{-1}_\gamma B.
\end{eqnarray*}
\end{subequations}

Then, the SDA algorithm presented in~\cite{GuoXULin2006} is given by
\begin{subequations}
\begin{eqnarray}
E_{k+1} &=& E_{k}(I_{n} - G_kH_k)^{-1}E_k,  \\
F_{k+1} &=& F_{k}(I_{n} - H_kG_k)^{-1}F_k,  \\
G_{k+1} &=& G_k+ E_{k}(I_{n} - G_kH_k)^{-1}G_kF_k,  \\
H_{k+1} &=& H_k + F_{k}(I_{n} - H_kG_k)^{-1}H_kE_k,
\end{eqnarray}
\end{subequations}
where the sequence $H_k$ and $G_k$ will converge to the minimal nonnegative solutions $X$ of~\eqref{NARE}  and $Y$ of~\eqref{DNARE}  quadratically.

\subsection{Spectrum Analysis}\label{sec_sda}
Recall that in the critical case $(\alpha,c) = (0,1)$, the matrix $M$~\eqref{Mmatrix}
is an irreducible singular M-matrix~\cite{Guo2001} and
the corresponding matrix $H $~\eqref{Hmatrix}
has a double zero eigenvalue with the geometric multiplicity equal to one.
To be specific, the matrix $H$ has $2n$ real eigenvalues ${\nu_n,\ldots,\nu_1,\lambda_1,...,\lambda_n}$,
which satisfy the following order~\cite{JuangLin1999}:
\begin{equation}
\frac{-1}{\omega_n}\! <\!  \nu_n \! <\!
\frac{-1}{\omega_{n-1}}\! <\! \ldots \! <\!
\frac{-1}{\omega_2}\! <\! \nu_2 \! <\!
\frac{-1}{\omega_1} \! <\!   \nu_1 \! =\!  0 \! =\!  \lambda_1\!  <
\frac{1}{\omega_1} \! <\! \lambda_2\! <\! \frac{1}{\omega_2}
\! <\! \ldots\! <\! \lambda_n\! <\! \frac{1}{\omega_n} .
 \end{equation}
\noindent The phenomenon is called \emph{eigenvalue interlacing}.
Moreover,
\begin{subequations}\label{HPEC}
\begin{eqnarray}
\sigma(D-C X) &=& \{\lambda_2, \ldots,\lambda_{n}, 0\},\\\sigma(A-B Y) &=& \{0, -\mu_{1}, \ldots,-\mu_{n}\},
\end{eqnarray}
\end{subequations}
if
$X$ and $Y$ are the minimal nonnegative solutions of~\eqref{NARE} and~\eqref{DNARE}, respectively~\cite{Guo2001}.
Paralleling the above distribution, the following theorem shows that all eigenvalues of $M$ are real and nonnegative. In fact, $M$ has $n$ specific eigenvalues $\frac{1}{\omega_i}$, for $i = 1,\ldots,n$.
\begin{theorem}\label{thm:schur}
Let $M$ be the matrix defined in \eqref{Mmatrix} with $(\alpha ,c)=(0,1)$.
Then $M$ has $2n$ real
eigenvalues, where one part of the eigenvalues of $M$ are  $0,\frac{1}{\omega_1},\ldots,\frac{1}{\omega_n}$ and the others are $\mu_1,\ldots,\mu_{n-1}$ such that the eigenvalues can be arranged in the following order:
 \begin{equation*}
    0<\frac{1}{\omega_1}< \mu_1 < \frac{1}{\omega_2} < \mu_2 <\cdots<\mu_{n-1}<\frac{1}{\omega_{n}}.
 \end{equation*}
\end{theorem}
\begin{proof}
Consider the characteristic polynomial of $M$ defined by
\begin{eqnarray}
  f(\lambda)&\equiv& \det(M-\lambda I_n) =\det( \bb \Gamma-\lambda I_n&\\ & \Delta-\lambda I_n \eb-\bb q\\e \eb \bb
e^\top & q^\top\eb) \nonumber\\
 &=& \det(\bb (\Gamma-\lambda I_n)&\\ & (\Delta-\lambda I_n) \eb)(1-\bb
e^\top & q^\top\eb \bb (\Gamma-\lambda I_n)^{-1}&\\ & (\Delta-\lambda I_n)^{-1} \eb  \bb q\\e \eb)  \nonumber \\
 &=& \prod_{1\leq i \leq n}(\gamma_i-\lambda)(\delta_i-\lambda)(1-\sum_{1\leq j\leq n}(\frac{q_j}{\gamma_j-\lambda}+\frac{q_j}{\delta_j-\lambda})) \label{secularM}
\end{eqnarray}
The last equation~\eqref{secularM} is called the secular equation of
$M-\lambda I$. Notice that $\gamma_i=\delta_i=\frac{1}{\omega_i}$,
$q_i=\frac{c_i}{2\omega_i}$ for $1\leq i \leq n$ when $(\alpha, c)=(0,1)$. Thus, through a straightforward calculation, we have
\begin{align*}
f(\lambda)&=\prod_{1\leq i \leq n}(\frac{1}{\omega_i}-\lambda)^2(1-\sum_{1\leq j\leq n}\frac{c_j}{(1-\omega_j\lambda)})\\
&=-\prod_{1\leq i \leq n}(\frac{1}{\omega_i}-\lambda)^2\sum_{1\leq
j\leq n}\frac{c_j\lambda}{\frac{1}{\omega_j}-\lambda}\\
 &=-\lambda\prod_{1\leq i \leq n}(\frac{1}{\omega_i}-\lambda) \left (\sum_{1\leq j\leq n}c_j \prod_{k\neq j, 1\leq k\leq n}
 (\frac{1}{\omega_k}-\lambda)\right).
\end{align*}
Thus, $f$ has roots
$0$,$\frac{1}{\omega_1},\cdots,\frac{1}{\omega_n}$. To complete the
proof of the theorem, let
\begin{align*}
g(\lambda)=\sum_{1\leq j\leq n}c_j \prod_{k\neq j, 1\leq k\leq n}
 (\frac{1}{\omega_k}-\lambda).
\end{align*}
The sign of $g(\frac{1}{\omega_j})$ is $(-1)^{j-1}$ since the
monotonicity of $\{ \omega_j \}$ , the intermediate value theorem
indicates that g has at least roots in
$(\frac{1}{\omega_j},\frac{1}{\omega_{j+1}})$ for $1\leq j \leq
n-1$. Together with the fact that the degree of $g$ is $n-1$. The
proof of the theorem is thus complete.
\end{proof}

It should be noted that
\begin{equation}\label{HD2}
 {H}\bb I_n\\ {X}\eb = \bb I_n\\{X}\eb (D - {C}{X}).
\end{equation}
From the above theorem and~\eqref{HD2}, we know that the minimal nonnegative solution $X$ is related to an invariant subspace with nonnegative eigenvalues of $H$. Also,
it is clear that  $q^{\top}\Gamma^{-1}e + e^{\top}\Delta^{-1}q = c = 1$. We then have the fact~\cite{Bini2008} that
the matrix $H$  has \emph{a right eigenvector}
$v^{\top} = [v_1^{\top}, v_2^{\top}]$, with $v_1 = \Gamma^{-1}q$
and $v_2 = \Delta^{-1}e$, so that
\begin{equation}
Hv = 0.
\end{equation}

By applying this right eigenvector $v$, \emph{a left eigenvector} $u^{\top} = [u_1^{\top}, u_2^{\top}]$, with $u_1 = \Gamma^{-1}e$ and $u_2 = -\Delta^{-1}q$ of $H$, corresponding to the  eigenvalue 0,  can be obtained without any trouble by directly checking that
\begin{equation}
u^{\top}H =  0 .
\end{equation}
Corresponding to the matrix $H$, the matrix $M = J H$ has the right and left eigenvectors $v$ and $u^{\top}J$.
Also, it can be seen that $ u_1^{\top}v_1 + u_2^{\top}v_2 = 0$. Applying the eigenpair information, we have  the following important result given in~\cite{Guo2001,GuoBruMei2007}.
\begin{theorem}\label{thm:nullrecurrent}
Let $M$ be an irreducible singular M-matrix as defined in \eqref{Mmatrix}, and let
$X$ and  $Y$ are the minimal nonnegative solutions of \eqref{NARE} and \eqref{DNARE}, respectively. Suppose that corresponding to the zero eigenvalue,  the right and left eigenvectors of $M$
are $v^{\top} =[v_1^{\top},v_2^{\top}] $ and $u^{\top} = [u_1^{\top}, -u_2^{\top}]$. If
$(\alpha, c) = (0,1)$,
then the following properties are satisfied:
\begin{eqnarray}\label{nullrecurrent}
Xv_1 &=& v_2, \quad u_2^{\top} X = - u_1^{\top},\quad \mbox{and}\quad Yv_2 = v_1.
\end{eqnarray}

\end{theorem}
It was shown in~\cite{GuoBruMei2007}, that
the matrix $X$ is the minimal nonnegative solution of~\ref{NARE} if and only if $X^{\top}$  is the minimal nonnegative solution of the equation
\begin{equation}\label{TNARE}
X^{\top}C^{\top}X^{\top} - X^{\top}A^{\top}-D^{\top}X^{\top}+B^{\top} =0.
\end{equation}
The same statement can be applied to the dual equation~(\ref{DNARE}). Its proof is simply based on
taking the transpose on both sides of ~(\ref{NARE}).

\begin{corollary}\label{cor:tdnare}
The matrix $Y$ is the minimal nonnegative solution of~\eqref{DNARE} if and only if $Y^{\top}$  is the minimal nonnegative solution of the equation
\begin{equation}\label{TDNARE}
Y^{\top}B^{\top}Y^{\top} - Y^{\top}D^{\top}-A^{\top}Y^{\top}+C^{\top} =0.
\end{equation}
\end{corollary}

Following Corollary~\eqref{cor:tdnare}, we want to know that whether there exists a relationship between the left eigenvector  of $M$ and the minimal nonnegative solution $Y$. To begin with, let
\begin{equation}
M_{t} = \left [
\begin{array}{cc}
D^{\top} & -B^{\top}\\
-C^{\top} & A^{\top}
\end{array} \right ],
\end{equation}
be the corresponding M-matrix of~\eqref{TDNARE}.
Note that
$M_{t}$ has
a right eigenvector
$[u_1^{\top}, -u_2^{\top}]^{\top}$ and a left eigenvector $ [v_1^{\top},v_2^{\top}]$ corresponding to the eigenvalue $0$.
 Equipped with the notations given in~(\ref{Mpara}),  the matrix
 $M_{t}$ is again an irreducible singular M-matrix if $(\alpha, c)=(0,1)$.  Then,
 Theorem~\ref{thm:nullrecurrent}  asserts that
  $Y^{\top}u_1 = -u_2$.  Namely, we have derived
  the following important relationship between the left eigenvector $u$ and the minimal solution $Y$,
 \begin{equation}\label{thm:nullrecurrent2}
u_1^{\top}Y = -u_2^{\top}.
\end{equation}


On the other hand,
we know that the convergence rate of the SDA algorithm is determined by
\begin{subequations}\label{covSDA}
\begin{eqnarray}
\limsup\limits_{k\rightarrow\infty} \sqrt[2^k]{\|H_{k}-X\|}
\leq \rho(\emph{C}_\gamma(D-CX))\rho(\emph{C}_\gamma(A- BY)),\label{covSDA1}
\\
\limsup\limits_{k\rightarrow\infty} \sqrt[2^k]{\|G_{k}-Y\|}
\leq \rho(\emph{C}_\gamma(D-CX))\rho(\emph{C}_\gamma(A- BY)),\label{covSDA2}
\end{eqnarray}
\end{subequations}
where
\begin{equation}
\emph{C}_\gamma: z \rightarrow \frac{z-\gamma}{z+\gamma}
\end{equation}
is the Cayley transform
and the scalar $\gamma> 0$~\cite{GuoBruMei2007}.
Note that from~\eqref{HPEC}, we have  $\rho(\emph{C}_\gamma(D-CX) = \rho(\emph{C}_\gamma(A- BY)) = 1$. It  follows that no further conclusion
of the convergence rate of the SDA algorithm
can be derived from the fact~\eqref{covSDA} except that the linear convergence is guaranteed. In the subsequent section, we want to know that how the shift procedures affect the convergence rate.


\section{Properties of the Shifted NARE}
In this section,  a detailed analysis of the eigenvalue distribution of the matrix $M$ is provided with respect to the the critical case $(\alpha, c) = (0,1)$. It is shown that under the shifting technique,
the matrix $M$ is still an M-matrix and the SDA algorithm is guaranteed to converge. The minimal nonnegative solution in the
shifted NARE problems are proved to be equal to the minimal nonnegative solution of~\eqref{NARE}. Last but not least,
the SDA algorithm is shown to be accelerated by removing the singularities embedded in the matrix $H$.

\subsection{Single Shift}\label{sec:singleshift}
Let $\widehat{H}$ be the rank-one modification of the matrix $H$ 
which is
defined by
 \begin{equation}\label{Hhat}
 \widehat{H} = H +\eta vr^{\top},
 \end{equation}
  where $\eta > 0$ is a scalar
  and $r\geq 0$ is a vector satsifying $r^{\top}v = 1$. To be specific, we write
  $r^{\top} = [r_1^{\top}, r_2^{\top}]$, where
  $r_1 = e$, $r_2 = q$.  Then, two matrices $\widehat{H}$ and $\widehat{M}$ are denoted by
  \begin{equation}\label{HMhat}
  \widehat{H} = \left [
\begin{array}{cc}
\widehat{D} & -\widehat{C}\\
\widehat{B} & -\widehat{A}
\end{array} \right ], \quad
  \widehat{M} = \left [
\begin{array}{cc}
\widehat{D} & -\widehat{C}\\
-\widehat{B} & \widehat{A}
\end{array} \right ],
 \end{equation}
 where
 \begin{eqnarray}\label{singalshift}
\widehat{D} = D + \eta v_1r_1^{\top}, \quad
\widehat{C} = C - \eta v_1r_2^{\top},\nonumber\\
\widehat{B} = B + \eta v_2r_1^{\top},\quad
\widehat{A} = A - \eta v_2r_2^{\top}.
 \end{eqnarray}

It follows from the specific structure of $\widehat{M}$ given in~\eqref{HMhat} that the matrix $\widehat{M}$ is irreducible.
The nice feature of this rank-one modification is that one zero eigenvalue of $H$ will be replaced by the scalar $\eta >0$. This can be seen by directly applying the following
useful lemma shown  in~\cite{GuoBruMei2007}.
\begin{lemma}\label{lem:shift}
Let $T$ be a singular matrix with $Tv = 0$ for some nonzero vector $v$. If $r$ is a vector so that $r^{\top}v =1$, then for any scalar r, the eigenvalues of the matrix
\begin{equation*}
\widehat{T} = T+ \eta v r^{\top},
\end{equation*}
consist of those of $T$, except that one zero eigenvalue of $T$ is replaced by $\eta$.
\end{lemma}

It can be seen that from Lemma~\ref{lem:shift}  the eigenvalues of $H$ and $\widehat{H}$ are the same except that one zero eigenvalue is shifted to $\eta$. In the next theorem, we want to show that despite of the rank one modification,
the eigenvalues of $\widehat{M}$ are equal to those of $M$.
\begin{corollary}\label{MhatMatrix}
Let $M$ and $\widehat{M}$ be defined in~\eqref{Mmatrix} and~\eqref{HMhat}, respectively. Then,
the characteristic polynomials of $M$ and $\widehat{M}$ are conincident. That is,  the eigenvalues of $M$ and $\widehat{M}$ are equal.
\end{corollary}

\begin{proof}
This proof can be easily obtained by studying the characteristic polynomial of $\widehat{M}$. We know that
the characteristic polynomial of $\widehat{M}$, denoted by
$\widehat{f}(\lambda)$, is
defined by

\begin{eqnarray}\label{charMhat}
  \widehat{f}(\lambda)&\equiv& \det(\widehat{M}-\lambda I_{2n}) =\det( \bb \Gamma-\lambda I_n&\\
  &   \Delta-\lambda I_n \eb  +
  \bb (- I_n + \eta \Gamma^{-1})q\\
  (- I_n - \eta \Delta^{-1})e
\eb
\bb
e^\top & q^\top
\eb)
\nonumber\\
 &=& \prod_{i =1}^n(\frac{1}{\omega_i}-\lambda)^2
  \det(
  1 +
  \bb
e^\top & q^\top
\eb
\bb (\Gamma-\lambda I_n)^{-1}&\\
  &   (\Delta-\lambda I_n)^{-1} \eb
  \bb (- I_n + \eta \Gamma^{-1})q\\
  (- I_n - \eta \Delta^{-1})e
\eb
 )  \nonumber \\
&= &-\prod_{i = 1}^n(\frac{1}{\omega_i}-\lambda)^2\sum_{j = 1}^n\frac{c_j\lambda}{\frac{1}{\omega_j}-\lambda}.
 \end{eqnarray}
From (\ref{charMhat}), we know that the eigenvalues of $\widehat{M}$ are precisely those of $M$.
\end{proof}


\begin{theorem}\label{MhatZMatrix}
The matrix
$\widehat{M}$ defined by equation~\eqref{HMhat}
is a Z-matrix if and only if
the parameter
$\eta$, defined in~\eqref{dshift} satisfy
\begin{eqnarray}  \label{sconstraint1}
 0 & < &  \eta \leq \frac{1}{\omega_1} .
\end{eqnarray}
\end{theorem}
\begin{proof}
From~\eqref{HMhat}, $\widehat{M}$ is a Z-matrix if and only if
$\widehat{B} \geq 0$, $\widehat{C} \geq 0$, and $\widehat{D}$ and $\widehat{A}$ are Z-matrices. Note that
\begin{align}
\begin{array}{rclrcl}
\widehat{D} &=& \Gamma + (-I_n + \eta \Gamma^{-1})qe^{\top},&
\widehat{C}&
= &(I_n - \eta \Gamma^{-1})qq^{\top} ,\\
\widehat{B} &=& (I + \eta \Delta^{-1})ee^{\top} > 0, &
\widehat{A} &=& \Delta + (-I_n - \eta \Delta^{-1})eq^{\top} .
\end{array}
\end{align}
The sufficient and necessary condition such that the matrix $\widehat{M}$ is a Z-matrix  is that
$\widehat{C} \geq0$, and $\widehat{D}$ and $\widehat{A}$ are Z-matrices.
This implies that
\begin{eqnarray}
-1 + \eta \omega_1& \leq & 0.
\end{eqnarray}\label{NsingC1}
Since $\eta$ is positive, we have the fact that  $\widehat{M}$ is a Z-matrix if and only if~\eqref{sconstraint1}
is satisfied.
\end{proof}

Using Corollary~\ref{MhatMatrix} and the given constraint~\eqref{sconstraint1} in
Theorem~\ref{MhatZMatrix}, we know that $\widehat{M}$ is an irreducible M-matrix and the SDA algorithm is guaranteed to be applicable.
It is known that the minimal nonnegative solution $\widehat{X}$
of the single shifted NARE is equivalent to the minimal nonnegative solution $X$ of~\eqref{NARE}~\cite{GuoBruMei2007}.
Thus, we have
\begin{eqnarray}\label{covSDAhat}
\limsup\limits_{k\rightarrow\infty} \sqrt[2^k]{\|H_{k}-X\|}
\leq \rho(\emph{C}_\gamma(\widehat{D}-\widehat{C}X))\rho(\emph{C}_\gamma(\widehat{A}- \widehat{B}\widehat{Y}))<1,
\end{eqnarray}
since $\rho(\emph{C}_\gamma(\widehat{D}-\widehat{C}X)
< 1$ and
$\rho(\emph{C}_\gamma(\widehat{A}- \widehat{B}\widehat{Y})) = 1$. It concludes that  the convergence of the SDA algorithm with a single shift is faster than that with no shift.
%
%
Based on all the properties stated above, it is illuminating to begin the analysis of the double shifting technique.
\subsection{Double Shift}
In order to remove all zero eigenvalues of $H$, we  define the double shifted matrix $\overline{H}$,
\begin{equation}\label{dshift}
\overline{H} = H + \eta v r^{\top} + \xi s u^{\top}=\bb \overline{D}&-\overline{C}\\\overline{B}&-\overline{A}\eb,
\end{equation}
where $\eta > 0$, $\xi < 0$, $p^{\top}$  and $q^{\top}$ such that $p^{\top}v = q^{\top}u = 1$, each size of sub-matrices $\overline{A},\overline{B},\overline{C}$ and $\overline{D}$ are $n$ square. This is the so called \emph{double shifting} technique. Indeed, it can be seen that if we choose $s^{\top} = [s_1^{\top},s_2^{\top}]$ with $s_1 = q$ and $s_2 = -e$  and the same vectors $r$, $u$ and $v$ as defined above, then the vectors $p$ and $q$ satisfy
the fact that
\begin{equation}
r^{\top}v = s^{\top}u = e^{\top}\Gamma^{-1}q + q^{\top}\Delta^{-1}e=1.
\end{equation}
From Lemma~\ref{lem:shift}, we know that the double shifting technique will move one zero eigenvalue of $H$ to $\eta>0$ and the other to $\xi < 0$ and keep the nonzero eigenvalues unchanged. With this in mind, the shift technique introduced in formula~(\ref{dshift}) will make the new matrix $\overline{H}$ nonsingular. Also, we can define a duble shifted NARE in $\overline{X}\in\rnn$ associate with the matrix $\overline{H}$ as follows:
\begin{subequations}
\begin{equation} \label{new NARE}
\overline{X}\,\overline{C}\, \overline{X}- \overline{X}\,\overline{D} - \overline{A} \,\overline{X} + \overline{B} = 0,
\end{equation}
and the dual duble shifted NARE in $\overline{Y}\in\rnn$
\begin{equation} \label{new DNARE}
\overline{Y}\,\overline{B}\, \overline{Y} - \overline{Y}\,\overline{A} - \overline{D} \,\overline{Y} + \overline{C} = 0,
\end{equation}
\end{subequations}
where
\begin{align}\label{barsubmat}
\begin{array}{c}
\overline{D} = D + \eta v_1r_1^{\top} + \xi s_1u_1^{\top}, \quad
\overline{C} = C - \eta v_1 r_2^{\top} - \xi s_1 u_2^{\top},\\
\overline{B} = B + \eta v_2 r_1^{\top} + \xi s_2 u_1^{\top}, \quad
\overline{A} = A - \eta v_2r_2^{\top} - \xi s_2 u_2^{\top}.
\end{array}
\end{align}

%
%

In what follows, we show that under suitable assumptions on parameters $\eta$ and $\xi$, the matrix $\overline{M}$ defined by
\begin{equation}\label{Mbar}
\overline{M} = \left [
\begin{array}{cc}
\overline{D} & -\overline{C}\\
-\overline{B} & \overline{A}
\end{array} \right ],
\end{equation}
is a nonsingular M-matrix, that is, the SDA algorithm is well-defined and applicable to the NARE~(\ref{new NARE}). We start our proof by showing that this matrix $\overline{M}$ is a Z-matrix for some parameters $\eta$ and $\xi$.

\begin{theorem} \label{thm:zmatrix}
The matrix
$\overline{M}$ defined by equation~\eqref{Mbar}
is a Z-matrix if and only if
the parameters,
$\eta$ and $\xi$, defined in~\eqref{dshift} satisfy
the following two conditions:

\begin{subequations}\label{constraints}
\begin{eqnarray}
 0 &<&  \eta < \frac{1}{\omega_1} ,
 \label{constraint1}
 \\
 \frac{-1+\eta \omega_1}{\omega_1}   &\leq& \xi <  0
 \label{constraint2}.
\end{eqnarray}
\end{subequations}

\end{theorem}
\begin{proof}
It follows from~(\ref{Mbar}) we know that $\overline{M}$ is a Z-matrix if and only if
$\overline{B} \geq 0$, $\overline{C} \geq 0$, and $\overline{D}$ and $\overline{A}$ are Z-matrices. Also, from
~(\ref{barsubmat}) we have
\begin{eqnarray*}
\overline{D} &=& \Gamma + (-I_n + \eta \Gamma^{-1})qe^{\top}+\xi qe^{\top}\Gamma^{-1}, \\
\overline{C}
&=& (I_n - \eta \Gamma^{-1})qq^{\top} + \xi q q^{\top}\Delta^{-1},\\
\overline{B} &=& (I + \eta \Delta^{-1})ee^{\top} - \xi e e^{\top} \Gamma^{-1} > 0,\\
\overline{A}&=& \Delta + (-I_n - \eta \Delta^{-1})eq^{\top} - \xi eq^{\top}\Delta^{-1}.
\end{eqnarray*}
Therefore, in order to get a Z-matrix $\overline{M}$, we only need to consider when
$\overline{C} \geq0$, and $\overline{D}$ and $\overline{A}$ are Z-matrices.
%
This gives rise to the following
three sufficient and necessary conditions:
\begin{align}\label{NSC1}
\left \{
\begin{array}{rcl}
-1 + \eta \omega_1 + \xi \omega_n & \leq & 0, \\
- 1+ \eta \omega_1 - \xi \omega_1&\leq &0, \\
-1 - \eta \omega_n - \xi \omega_1 & \leq& 0.
\end{array}
\right .
\end{align}
It follows from ~\eqref{NSC1} and the initial conditions  $\eta > 0$ and $\xi <0$
that $\overline{M}$ is a Z-matrix if and only if~\eqref{constraint1}
and~\eqref{constraint2} are satisfied.
\end{proof}

%

To simplify our discussion, we define
\begin{equation}\label{set:omega}
\Omega=\{(\eta,\xi); 0 <  \eta <
\frac{1}{\omega_1} , \frac{-1+\eta \omega_1}{\omega_1}  \leq \xi <  0\}.
\end{equation}
Our next approach is to show that the matrix $\overline{M}$ is indeed an M-matrix. That is,
the iterative processes in SDA algorithm do not break down
and convergence quadratically. To begin with, we introduce the following two lemmas.
\begin{lemma}\label{gorder1}
Let $c_i$ and $\omega_i$, for $i= 1,\ldots,n$, be defined in
~\eqref{NARE}.  Given $\lambda \in \mathbb{R}$
and $\lambda \neq \frac{1}{\omega_i}$, for $i=1,\ldots,n$, we define
\begin{equation}\label{g123}
g_1(\lambda) = \lambda\sum_{i=1}^n \frac{c_i}{\frac{1}{\omega_i} -\lambda},\quad
g_2(\lambda) = \sum_{i=1}^n \frac{c_i\omega_i}{\frac{1}{\omega_i} -\lambda},\quad
g_3(\lambda) = \sum_{i=1}^n \frac{c_i}{\omega_i(\frac{1}{\omega_i} -\lambda)}.
\end{equation}
Then, the following properties hold:
\begin{enumerate}
\item $
 g_1(\lambda) - \lambda^2 g_2(\lambda)   = \lambda \sum\limits_{i=1}^n c_i\omega_i
$.
\item $g_1(\lambda)  - g_3(\lambda)  = -1$.
\item If $\lambda \in (\frac{1}{\omega_k},\frac{1}{\omega_{k+1}})$, then $g_3(\lambda) \geq g_1(\lambda)\frac{1}{\lambda \omega_k}\geq g_2(\lambda)\frac{1}{\omega_k}.$
\end{enumerate}
\end{lemma}
\begin{proof}
The first two properties are following from the direct computation. To see this, applying the conditions in~\eqref{g123}, we have
\begin{eqnarray*}
g_1(\lambda) - \lambda^2 g_2(\lambda) & = &
\lambda \sum_{i=1}^n
\frac{( c_i- \lambda c_i\omega_i )\omega_i}{\omega_i(\frac{1}{\omega_i} -\lambda)}\nonumber \\
&=&\lambda \sum\limits_{i=1}^n c_i\omega_i.
\end{eqnarray*}
\begin{eqnarray*}
g_1(\lambda) - g_3(\lambda)& = &
\sum_{i=1}^n
\frac{( \lambda c_i\omega_i  - c_i)}{\omega_i(\frac{1}{\omega_i} -\lambda)} = -1.
\end{eqnarray*}

\noindent Using the triangle inequality and $\lambda \in (\frac{1}{\omega_k},\frac{1}{\omega_{k+1}})$,  we obtain
\begin{eqnarray*}
g_3(\lambda) & = &\sum_{i=1}^n \frac{c_i}{\omega_i(\frac{1}{\omega_i} -\lambda)}\nonumber \\
& \geq& \sum_{ i = 1}^k \frac{c_i}{\omega_k(\frac{1}{\omega_i} -\lambda)} + \sum_{i = k+1}^n \frac{c_i}{\omega_{k+1}(\frac{1}{\omega_i} -\lambda)}\nonumber \\
&\geq& g_1(\lambda)\frac{1}{\lambda \omega_k}\nonumber \\
& \geq& \left(\sum_{i = 1}^k \frac{c_i\omega_k}{(\frac{1}{\omega_i} -\lambda)} + \sum_{i = k+1}^n \frac{c_i\omega_{k+1}}{(\frac{1}{\omega_i} -\lambda)}\right)\frac{1}{\omega_k}\nonumber \\
&\geq& g_2(\lambda)\frac{1}{\omega_k}.
\end{eqnarray*}
\end{proof}

We have now seen that the relationships among $g_1(\lambda)$, $g_2(\lambda)$, and $g_3(\lambda)$. Let $g(\lambda)$ to be a function satisfying
\begin{eqnarray}\label{glambda}
g(\lambda)
& \equiv & \lambda g_1(\lambda) + \eta \xi g_2(\lambda) g_3(\lambda),
\end{eqnarray}
where $(\eta,\xi)\in\Omega$. Our next approach is to show that
for each subinterval $(\frac{1}{\omega_k},\frac{1}{\omega_{k+1}})$  with  $k = 1,\ldots, n-1$, there exists a point $\lambda$ so that
$g(\lambda) > 0$.  This property is
a stepping stone for showing that
$\overline{M}$ is an M-matrix.

\begin{lemma}\label{ChiangPf}
Let $c_i$ and $\omega_i$, for $i= 1,\ldots,n$, be defined in~\eqref{NARE}. It then follows that there exists a point $\lambda_k \in (\frac{1}{\omega_k},\frac{1}{\omega_{k+1}})$, for  $k = 1,\ldots, n-1$, so that
 the function $g(\lambda)$ of \eqref{glambda} is greater than zero.
\end{lemma}
\begin{proof}
Note that $g_3(\lambda)$ is a continuous function on $ (\frac{1}{\omega_k},\frac{1}{\omega_{k+1}})$,
$\lim\limits_{\lambda\to
{\frac{1}{\omega_k}}^+} g_{3}(\lambda) = -\infty$
, and $\lim\limits_{\lambda\to
{\frac{1}{\omega_{k+1}}}^-} g_{3}(\lambda) = +\infty$, for all  $k = 1,\ldots, n-1$.
Thus, there exists a point $\lambda_k \in (\frac{1}{\omega_k},\frac{1}{\omega_{k+1}})$ such that
\begin{equation}
g_3(\lambda_k) = \frac{4\omega_1^2}{\omega_k\omega_{k+1}}.
\end{equation}
Since $\omega_1>\omega_2>\cdots>\omega_n$, we have the fact that $g_3(\lambda_k) > 4$.
It follows from Lemma~\ref{gorder1} that $g_1(\lambda_k) > 0$.

We first assume that
$g_2(\lambda_k) < 0$ for this specific $\lambda_k$, then it  is clear that $g(\lambda_k)  = \lambda_k g_1(\lambda_k) + \eta \xi g_2(\lambda_k) g_3(\lambda_k)
> 0$, since $\eta \xi <0$.  We now assume that $g_2(\lambda_k) >0$. Combining the inequalities~\eqref{constraint1} with~\eqref{constraint2}, we have
\begin{equation}
 -\frac{1}{4\omega_1^2    }\leq \eta \xi < 0.
\end{equation}
Then, by~\eqref{glambda} we get
\begin{eqnarray}
g(\lambda_k) & \geq& \sum_{i = 1}^n \frac{c_i
(\lambda_k - \frac{\omega_i}{\omega_k\omega_{k+1}})
}{\frac{1}{\omega_i} -\lambda_k} \nonumber \\
& =& \sum_{i = 1}^k \frac{c_i
(\lambda_k - \frac{\omega_i}{\omega_k\omega_{k+1}})
}{\frac{1}{\omega_i} -\lambda_k}
+ \sum_{i = k+1}^n \frac{c_i
(\lambda_k - \frac{\omega_i}{\omega_k\omega_{k+1}})
}{\frac{1}{\omega_i} -\lambda_k} \geq 0,
\end{eqnarray}
since $\lambda_k  - \frac{\omega_i}{\omega_k\omega_{k+1}} < \frac{1}{\omega_{k+1}}
- \frac{\omega_k}{\omega_k\omega_{k+1}} = 0$, for $1\leq i \leq k$,
and $\lambda_k  - \frac{\omega_i}{\omega_k\omega_{k+1}} >\omega_{k} - \frac{\omega_{k+1}}{\omega_k\omega_{k+1}} = 0$, for $k+1\leq  i \leq n$.

\end{proof}

Now we have enough tools to validate that the given matrix $\overline{M}$ is indeed an M-matrix. In particular, we can also dig out the eigenvalue distribution of matrix $\overline{M}$.

\begin{theorem}\label{Mbareig}
If
$(\eta,\xi)\in\Omega$,
then
 the matrix
$\overline{M}$ defined by equation~\eqref{Mbar}
is an M-matrix. In particular, $\overline{M}$ has $2n$ positive real eigenvalues $\overline{\lambda_1},\ldots,\overline{\lambda}_{2n}$
satisfying
\begin{equation}\label{eigMbar}
0 < \overline{\lambda}_1<\overline{\lambda}_2<\frac{1}{\omega_1}<\overline{\lambda}_3<\overline{\lambda}_4<\frac{1}{\omega_2}<\cdots<\frac{1}{\omega_{n-1}}
<\overline{\lambda}_{2n-1}<\overline{\lambda}_{2n}<\frac{1}{\omega_n}
\end{equation}

\end{theorem}
\begin{proof}
Since the matrix $\overline{H}$ of~\eqref{dshift} is nonsingular, it is clear that $\overline{M} = J\overline{H}$ is nonsingular.
Also, Theorem~\ref{thm:zmatrix} implies $\overline{M}$ is a Z-matrix. In order to show that $\overline{M}$ is an M-matrix, it suffices to show that all eigenvalues of $\overline{M}$ have positive real part.
Indeed, all eigenvalues of $\overline{M}$ are positive real numbers and satisfy the interlacing property.

We first consider the characteristic polynomial $\bar{f}(\lambda)$ of $\overline{M}$ defined by
\begin{eqnarray}
 \bar{f}(\lambda)&\equiv& \det(\overline{M}-\lambda I_{2n})\nonumber\\
 &=&
 \det( \bb \Gamma-\lambda I_n&\\
  &   \Delta-\lambda I_n \eb
 +
  \bb
(- I_n + \eta \Gamma^{-1})q & q\\
  (- I_n - \eta \Delta^{-1})e & e
  \eb
  \bb
  e^\top & q^\top\\
\xi e^{\top}\Gamma^{-1} &-\xi q^{\top}\Delta^{-1}
  \eb
  )
\nonumber\\
%
& = & -\prod_{i = 1}^n(\frac{1}{\omega_i}-\lambda)^2 g(\lambda)\label{CharMbar1}\\
& = &
-\sum_{i=1}^n c_j \prod_{
1\leq s\leq n, s\neq i
}(\frac{1}{\omega_i}-\lambda)
\sum_{k=1}^n c_k
\left [
(-\lambda^{2} + \frac{\lambda}{\omega_k}
+ \frac{\xi\eta\omega_k}{\omega_j}
)
\prod_{
1\leq s\leq n, s\neq k
}(\frac{1}{\omega_s}-\lambda)
\right ],\label{CharMbar2}
\end{eqnarray}
where $g(\lambda)$ is the function given in~(\ref{glambda}). By direct substitution of $\frac{1}{\omega_k}$ in~(\ref{CharMbar2}), we have $\bar{f}(\frac{1}{\omega_k})
> 0
$, for $k = 1,\ldots, n$. Also, it follows from~\eqref{CharMbar1} that  $\bar{f}(0) > 0$.
If
we can find a point $\lambda$ satisfying $\bar{f}(\omega) < 0$ in each subinterval  $(\frac{1}{\omega_k},\frac{1}{\omega_{k+1}})$, for $k = 1,\ldots, n$ and the interval $(0,\frac{1}{\omega_1})$, then
the intermediate value theorem imply that
the distribution of eigenvalues of $\overline{M}$ arranged in~(\ref{eigMbar}) is valid. This also gives rise to the fact that $\overline{M}$ is a nonsingular M-matrix.

Next, we consider the subinterval $(0,\frac{1}{\omega_1})$. Choosing $\lambda = \frac{1}{2\omega_1}$, it follows that
\begin{eqnarray}
g(\frac{1}{2\omega_1})
& = &
\frac{1}{2\omega_1}
\sum_{i=1}^n \frac{c_i}{\frac{1}{\omega_i} -\frac{1}{2\omega_1}}
+ \eta \xi \sum_{i=1}^n \frac{c_i\omega_i}{\frac{1}{\omega_i} -\frac{1}{2\omega_1}}\sum_{i=1}^n \frac{c_i}{\omega_i(\frac{1}{\omega_i} -\frac{1}{2\omega_1})}\nonumber\\
& \geq &
(c_1 +
\sum_{i=2}^n \frac{c_i\omega_i}{2\omega_{1} -\omega_i}
)
-
(c_1 +
\sum_{i=2}^n \frac{c_i\omega_i
(\frac{\omega_i}{\omega_1})
}{2\omega_{i} -\omega_i}
)
(c_1 +
\sum_{i=2}^n \frac{c_i\omega_1}{2\omega_{i} -\omega_i}
)\label{omega11}\\
&\geq& 0 \label{omega12}
\end{eqnarray}
The second inequality~\eqref{omega11} comes from the fact that $ \eta \xi \geq -\frac{1}{4\omega_1^2    }$. Also,
since $\frac{\omega_k}{\omega_1} < 1$ and
$\frac{ci\omega_1}{2\omega_1 - \omega_i}<c_i$, for $i = 2,\ldots,n$,
 and $\sum\limits_{i=1}^n c_i  =1$, we have the last inequality~\eqref{omega12}. For the proof of each subinterval $(\frac{1}{\omega_k},\frac{1}{\omega_{k+1}})$, we simply apply the conclusion of Lemma~\ref{ChiangPf}. Then, \eqref{CharMbar1} immediately implies that there exists a point $\lambda\in (\frac{1}{\omega_k},\frac{1}{\omega_{k+1}})$ such that $f(\lambda) < 0$, for $k = 1,\ldots, n$.

\end{proof}

Note that in~\cite{GuoBruMei2007} the minimal nonnegative solution $X$ of~\eqref{NARE} has been shown to be a solution of~\eqref{new NARE}.  So far, to the best of our knowledge, no study has investigated the relation between the solutions $\overline{X}$ and $X$.  If there does not exist any relation between $\overline{X}$ and $X$, the solution obtained from the duble-shift algorithm would be  exclusively meaningless. Our next result is to find this substantial link through the known fact that $\overline{M}$ is indeed an M-matrix~\eqref{Mbareig}.

%

\begin{theorem}\label{XbarXYbarY}
Let $\overline{X}$ and $X$ be the minimal nonnegative solutions of~\eqref{new NARE} and~\eqref{NARE}, respectively. Then,  $
\sigma(\overline{D}-\overline{C} X) = \{\eta, \lambda_2, \ldots,\lambda_{n}\}$
and
$\overline{X} = X$.

\end{theorem}

\begin{proof}
Let $\mathcal{R}(Z) = ZCZ - ZD - AZ + B$ and
$\overline{\mathcal{R}}(Z) = Z\overline{C} Z - Z\overline{D} - \overline{A} Z + \overline{B}$.
Observe first that
\begin{eqnarray}\label{Rbarx}
\overline{\mathcal{R}}(X) &=& \mathcal{R}(X)
-\eta (Xv_1- v_2)(r_2^{\top} X+ r_1^{\top}) + \xi (Xs_1 - s_2)(-u_2^{\top}X - u_1^{\top})
=  \mathcal{R}(X),
\end{eqnarray}
where the second equality follows directly from
Theorem~\ref{thm:nullrecurrent}. This equality amounts to say that the minimal nonnegative solution of (\ref{NARE}) is also a nonnegative solution of (\ref{new NARE}) and the following equality is satisfied.
\begin{equation}
\overline{H} \bb I_n\\X\eb = \bb I_n\\X\eb (\overline{D} - \overline{C}X).
\end{equation}
%
Recall that  $u_1^{\top} + u_2^{\top}X= 0$. Then, we have

\begin{eqnarray*}
(\overline{D} - \overline{C} X)  &=& D - CX + \eta v_1
(r_1^{\top} + r_2^{\top}X)+\xi s_1(u_1^{\top} + u_2^{\top}X)\\
&=&D - CX + \eta v_1 (r_1^{\top} + r_2^{\top}X).
\end{eqnarray*}
Together with the fact that
\begin{equation*}
(D-CX)v_1 = (\Gamma - q e^{\top})\Gamma^{-1}q - q q^{\top}\Delta^{-1}e = 0,
\end{equation*}
and
\begin{equation*}
(r_1^{\top} + r_2^{\top}X) v_1 = e^{\top}\Gamma^{-1}q + q^{\top}\Delta^{-1}e=1,
\end{equation*}
we obtain
\begin{equation*}
(\overline{D}-\overline{C} X)v_1=(\overline{D}-\overline{C} X)v_1 = \eta v_1.
\end{equation*}
Then, Lemma~\ref{lem:shift} and Theorem~\ref{thm:schur} imply that $\sigma(\overline{D} - \overline{C}X) = \{\eta,\lambda_2,\ldots,\lambda_{n}\}$.

Since $\overline{M}$ is a nonsingular M-matrix and $\overline{X}$
is the minimal nonnegative solution of~\eqref{new NARE}, Theorem~\ref{thm_mmatrix1} and Theorem~\ref{thm:prop1} imply that
$\sigma(\overline{D}-\overline{C}\mbox{ }\overline{X})\subset \mathbb{C}_+$.
With this in mind, we have
\begin{equation}\label{HbarEig}
\sigma(\overline{D} - \overline{C} \mbox{ }\overline{X})
=
\sigma(\overline{D} - \overline{C}X).
 \end{equation}
Note that
\begin{equation}\label{HbarDbar2}
 \overline{H}\bb I_n\\\overline{X}\eb = \bb I_n\\\overline{X}\eb (\overline{D} - \overline{C}\mbox{ }\overline{X}).
\end{equation}
By~\eqref{HbarEig} and~\eqref{HbarDbar2},
it is true that
\begin{equation*}
\mbox{span}\bb I_n\\X\eb=\mbox{span}\bb I_n\\\overline{X}\eb.
\end{equation*}
Then, 
there exists a nonsingular matrix $S\in\rnn$ such that
 \begin{equation*}
\bb I_n\\X\eb=\bb I_n\\\overline{X}\eb S.
\end{equation*}
It is clear that this nonsingular matrix $S$ is an identity matrix. So, we conclude that $X = \overline{X}$.

\end{proof}
From Theorem~\ref{Mbareig} and Theorem~\ref{XbarXYbarY}, we know that $\overline{M}$ is a nonsingular M-matrix. Then, the SDA algorithm is guaranteed to converge. Similar to the discussion given in the single shifted algorithm, we have
\begin{eqnarray}
\limsup\limits_{k\rightarrow\infty} \sqrt[2^k]{\|H_{k}-X\|}
\leq \rho(\emph{C}_\gamma(\overline{D}-\overline{C}X))\rho(\emph{C}_\gamma(\overline{A}- \overline{B}\overline{Y}))<1,
\end{eqnarray}
since $\rho(\emph{C}_\gamma(\overline{D}-\overline{C}X)
< 1$ and
$\rho(\emph{C}_\gamma(\overline{A}- \overline{B}\overline{Y})) < 1$.
This also implies that for any $(\eta,\xi)\in\Omega$, the SDA algorithm with double shifts converges faster than the SDA algorithm with no shift and is quadratically convergent.

\section{Advantages of the Shifting Technique Applied to SI}

In~\cite{LuSIMAX05}, Lu shows that the minimal nonnegative solution $X$ of~\eqref{NARE} must be of the form:
\begin{equation*}
X = T\circ( {m} {n}^{\top}) = ( {m} {n}^{\top})\circ T.
\end{equation*}
Here, the symbol $\circ$ is the Hadamard product,
$
T = (t_{ij}) = \left( \frac{1}{\delta_i + \gamma_j}\right),
$
and $( {m} , {n} )$  is satisfying the vector equation:
\begin{subequations}\label{SIxy}
\begin{align}
 {m} &=  {m} \circ (P {n} ) +  {e},\\
 {n} &=  {n}  \circ(Q {m}) +  {e},
\end{align}
\end{subequations}
with
\begin{equation}
P = (P_{ij})=  \left( \frac{q_j}{\delta_i + \gamma_j}\right), \quad
Q = (Q_{ij})=  \left( \frac{q_j}{\delta_j + \gamma_i}\right).
\end{equation}
The SI method for finding the minimal nonnegative solution $( {m} , {n} )$ is then given by
\begin{subequations}\label{SIiter}
\begin{align}
 {m}^{(k+1)} &=  {m}^{(k)} \circ (P {n}^{(k)}) +  {e},\\
 {n} ^{(k+1)} &=  {n} ^{(k)} \circ(Q {m}^{(k)}) +  {e}.
\end{align}
\end{subequations}

Our aim in this section is to discuss how the shifted approaches can speed up the SI method. Theoretical discussion is also given to analyze the convergence of the SI method with shift. We then rewrite the coefficient matrices~\eqref{barsubmat} as

\begin{subequations}\label{barsubmat1}
\begin{eqnarray}
\overline{D} = \overline{D} (\eta,\xi)&=& \Gamma-Q_1(\eta)E_1(\xi)^\top, \\
\overline{C} = \overline{C}(\eta,\xi)&=&  Q_1(\eta)Q_2(\xi)^\top,\\
\overline{B}= \overline{B} (\eta,\xi)&=& E_2(\eta)E_1(\xi)^\top ,\\
\overline{A}= \overline{A}(\eta,\xi)&=&  \Delta-E_2(\eta)Q_2(\xi)^\top,
\end{eqnarray}
\end{subequations}
with
\begin{align*}
\begin{array}{ll}
Q_1 = Q_1(\eta) =\bb (I_n - \eta \Gamma^{-1}) {q} &
 {q}\eb,
 &   Q_2 = Q_2(\xi) = \bb  {q} & \xi\Delta^{-1} {q}\eb,\\
 E_1 = E_1(\xi) = \bb  {e} &
-\xi\Gamma^{-1} {e} \eb,
& E_2= E_2(\eta) =\bb (I_n+\eta\Delta^{-1}) {e} &  {e} \eb,
\end{array}
\end{align*}
and
relax the boundary conditions $(\eta, \xi)$ so that
$(\eta,\xi)\in\bar{\Omega}$. Here, $\bar{\Omega}$ is the closure of the set $\Omega$ defined in~\eqref{set:omega}.
Substituting \eqref{barsubmat1} into \eqref{new NARE}, we have
\begin{equation}\label{newiterfrac1}
Z\Gamma+\Delta Z= (ZQ_1+E_2)(Q_2^\top Z+E_1^\top).\end{equation}
This implies that the minimal nonnegative solution $Z$ of~\eqref{new NARE} can be written as
\begin{equation}\label{eq:solZ}
Z=T\circ(MN^\top),
\end{equation}
with
$M=ZQ_1+E_2\in\mathbb{R}^{n\times 2},\,N^\top=Q_2^\top Z+E_1^\top\in\mathbb{R}^{2\times n}$.

%
Akin to the iteration given in~\eqref{SIiter}, the iteration sequence
 $\{M_k,N_k\}$  corresponding to~\eqref{eq:solZ} can be written as

\begin{subequations}\label{newiterUV}
\begin{align}
M_{k+1}&=(T\circ(M_kN_k^\top))Q_1+E_2,\\
N_{k+1}&=(T\circ(N_kM_k^\top))Q_2+E_1,
\end{align}
with the initial value
\begin{align}
M_0=0,N_0=0.
\end{align}
\end{subequations}
Let $Z_k(\eta,\xi) = Z_k = T\circ(M_k N_k^\top)$, for all $k$.
Corresponding to~\eqref{newiterfrac1}, we then have the classical fixed-point iteration,

\begin{equation}\label{classfixiter}
Z_{k+1}\equiv
T\circ \left((Z_k Q_1 + E_2)( Q_2^\top Z_{k} + E_1^\top)\right).
\end{equation}

Our next theorem is to show that the sequence $\{Z_k\}$ does indeed converge and converge to the minimal nonnegative solution $X$ of~\eqref{NARE}.
\begin{theorem}\label{CovF}
Assume that
\begin{equation}
 \overline{R}(X^*)  = X^*\overline{C} X ^*- X^*\overline{D} - \overline{A} X^* + \overline{B}\leq 0,
 \end{equation}
for some nonnegative matrix $X^*$.
Then for the fixed-point iteration~\eqref{classfixiter} with initial value $Z_0 = 0$, we have
\begin{equation}\label{OrdReL1}
Z_0 < Z_1 < \cdots < Z_k < X^*, \mbox{ for any } k\geq 1.
\end{equation}
Moreover, $\lim\limits_{k\rightarrow\infty} Z_k(\eta,\xi) = X$ for any
$(\eta,\xi)\in\Omega$.
\end{theorem}

\begin{proof}
By~\eqref{barsubmat1}, $Q_1E_1^{\top} \geq 0$, $Q_1Q_2^{\top}  = \overline{C} \geq 0$, $E_2E_1^{\top}  = \overline{B} \geq 0$, and $E_2Q_2 \geq 0$.
It follows that~\eqref{OrdReL1} holds by induction.
Since the sequence $\{Z_k\}$ is monotonically increasing and bounded above, we have $\lim\limits_{k\rightarrow\infty} Z_k = Z^*$, for some $Z^*$.
Hence $ \overline{R}(Z^*) = 0$. On the other hand, since $Z^*\leq X^*$ for any nonnegative matrix $X^*$, we have  $Z^* = X$.
\end{proof}

The convergence property, shown in Theorem~\ref{CovF}, is of fundamental importance in our subsequence discussion
and can induce the possibility of analyzing a number of convergent behaviors in the SI method with shift. Note that since $M_k$ and $N_k$ are matrices in $\mathbb{R}^{n\times 2}$, we can define
\begin{align}\label{newiter1}
M_k=\bb  {m}^{(k)}_1 &  {m}^{(k)}_2 \eb,\,\, N_k=\bb  {n}^{(k)}_1 &  {n}^{(k)}_2 \eb,
\end{align}
where $ {m}^{(k)}_1, {m}^{(k)}_2, {n}^{(k)}_1$ and $ {n}^{(k)}_2$ are $n$-dimension column vectors. It follows that we have the equivalent iteration for $Z_k$, that is,
\begin{equation}\label{eq:zkmn}
Z_k = T\circ\left( {m}^{(k)}_{1}( {n}^{(k)}_{1})^\top+ {m}^{(k)}_{2}( {n}^{(k)}_{2})^\top\right).
\end{equation}
Then, we obtain the new algorithm of the SI with shift, given by
\begin{subequations}\label{newiter0}
\begin{align}
 {m}^{(k+1)}_{1}&=Z_k(I_n-\eta\Gamma^{-1}) {q}+(I_n+\eta\Gamma^{-1}) {e},\\
  {m}^{(k+1)}_{2} &=Z_k {q}+ {e},\\
 {n}^{(k+1)}_{1} &=Z_k^\top  {q}+ {e},\\
 {n}^{(k+1)}_{2}
             &=-\xi(\Gamma^{-1} {e}-Z_k^\top\Delta^{-1} {q}).
\end{align}
\end{subequations}
with the initial value
\begin{subequations}\label{newiterinit}
\begin{align}
 {m}^{(0)}_{1}&=  {0},\, {m}^{(0)}_{2}=  {0},\\
 {n}^{(0)}_{1}&=  {0},\, {n}^{(0)}_{2}=  {0}.
\end{align}
\end{subequations}

It is true that this SI iteration with shift is still a method with a cost of $O(n^2)$ ops but requires more calculations than the original SI method.  However, in order to have a method with a better behavior, adding some complexity is sometimes a necessary sacrifice. Actually, we can  simplify our computation by consider the following iteration,
\begin{align*}
 {m}^{(k+1)}_{2}&=  \bb  Z_k & I_n \eb \bb  {q} \\  {e} \eb,\\
 {m}^{(k+1)}_{1}& = {m}^{(k+1)}_{2}+\eta\bb Z_k & I_n  \eb \bb -\Gamma^{-1} {q} \\ \Delta^{-1} {e}\eb,\\
 {n}^{(k+1)}_{1}& =\bb  Z_k^\top & I_n  \eb \bb  {q} \\  {e}\eb,\\
 {n}^{(k+1)}_{2}& =-\xi \bb Z_k^{\top} & I_n  \eb \bb -\Delta^{-1} {q} \\ \Gamma^{-1} {e}\eb.
\end{align*}
%
%

In next theorem, we discuss the convergent property of  the sequence ${( {m}^{(k)}_1,  {m}^{(k)}_2,  {n}^{(k)}_1,
 {n}^{(k)}_2)}$ and the convergent speed of the sequence ${Z_k}$.
\begin{theorem}\label{newiterthm}
Given $(\alpha, c) = (0, 1)$, the sequence ${( {m}^{(k)}_1,  {m}^{(k)}_2,  {n}^{(k)}_1,
- {n}^{(k)}_2)}$ with initial values\eqref{newiterinit} is strictly monotonically increasing and satisfies the following two conditions:
 \begin{itemize}
\item[a.]
 $ {e}\leq  {m}_1^{(k)}\leq  {m}$, $ {e}\leq  {m}_2^{(k)}\leq  {m}$, $ {e}\leq  {n}_1^{(k)}\leq  {n}$, $0\leq  {n}_2^{(k)}\leq -\xi\Gamma^{-1} {e}$.

\item[b.]
$\lim\limits_{k\rightarrow\infty}  {m}^{(k)}_1=\lim\limits_{k\rightarrow\infty}  {m}^{(k)}_2= {m}$, $\lim\limits_{k\rightarrow\infty}  {n}^{(k)}_2= {n}$, $\lim\limits_{k\rightarrow\infty}  {n}^{(k)}_2= {0},$
\end{itemize}
where $ {m}$ and $ {n}$ are defined on \eqref{SIxy}. In fact, in the critical case, we have $ {m}= {n}$ and $X=X^\top$.
\end{theorem}

\begin{proof}
From Theorem~\ref{CovF}, we know that $Z_0 <
Z_1<\ldots<Z_k <X$ and $\lim\limits_{k\rightarrow \infty} Z_k = X$.  Substituting these two facts to~\eqref{newiter0}, we immediately have
\begin{align*}
\begin{array}{rclclclclcl}
 {e}  &\leq&  {m}^{(1)}_{2} &<&  {m}^{(2)}_{2} &< &\ldots &< & {m}^{(k)}_{2} &\leq&  X {q}+ {e} =  {m},\\
 {e} & \leq & {m}^{(1)}_{1} &< & {m}^{(2)}_{1} &< &\ldots &<&  {m}^{(k)}_{1} &\leq&  X {q}+ {e} =  {m},\\
 {e} & \leq &  {n}^{(1)}_{1} & <&  {n}^{(2)}_{1} &< &\ldots &< & {n}^{(k)}_{1} &\leq& X^{\top}  {q} +  {e} =  {n},\\
-\xi \Gamma^{-1} {e} & \geq &   {n}^{(1)}_{2} &>&   {n}^{(2)}_{2}&>& \ldots& > & {n}^{(k)}_{2} &\geq & -\xi(\Gamma^{-1} {e} - X^{\top} \Delta^{-1} {q} )  =  {0}.
\end{array}
\end{align*}
Note that the order of the sequence$\{ {m}_1^{(k)}\}$ comes from the fact that $Z_k  {q} - \eta Z_k
\Gamma^{-1}  {q}  >  0$ and the last equality of  the sequence
$\{ {n}_2^{(k)}\}$  comes from Theorem~\ref{thm:nullrecurrent}.

From~\eqref{eq:zkmn}, part (b) holds, since $\lim\limits_{k\rightarrow \infty} Z_k = X$.
\end{proof}

When we studied the shifted procedures, our main purpose is to speed up the convergence. In what follows we discuss the relations of $Z_k(\eta,\xi)$ with respect to different $\eta$ and $\xi$ values and show that the SI with shift converges linear, instead of  sublinear.

%
%
%

\begin{theorem}\label{SITHM}
Given $(\alpha, c) = (0, 1)$, the sequence $\{Z_k\}$ has the following two properties:

 \begin{itemize}
\item[a.]
$ Z_k(0,0)\leq Z_k(\eta,0)\leq Z_k(\eta,\xi), \mbox{ for each } k \mbox{ and }(\eta,\xi)\in\Omega$.

\item[b.] The sequence $\{Z_k(\eta,\xi)\}$ converges linearly to the minimal nonnegative solution $X$ of \eqref{NARE}
 for all $(\eta,\xi)\in\Omega$.
\end{itemize}

\end{theorem}
\begin{proof}
From~\eqref{eq:solZ}, we have
\begin{align*}
Z&=T\circ \left(ZQ_1(\eta)Q_2(\xi)^\top Z+E_2(\eta)Q_2(\xi)^\top Z+ZQ_1(\eta)E_1(\xi)^\top+E_2(\eta)E_1(\xi)^\top\right) \\
&=T\circ \left((ZQ_1(0)Q_2(0)^\top Z+E_2(0)Q_2(0)^\top Z+ZQ_1(0)E_1(0)^\top+E_2(0)E_1(0)^\top \right)\\
  &+\eta T\circ \left((-Z\Gamma^{-1}q+\Delta^{-1}e)(q^\top Z+e^\top)\right)-\xi T\circ \left( (Zq+e)(-q^\top\Delta^{-1}Z+e^\top\Gamma^{-1})\right).
\end{align*}
Subsequently, it follows from mathematical induction that part (a) holds.

For the proof of part (b), we need to use three well-known results discussed in~\cite{GuoLaub2000}.
First, for the iteration~\eqref{classfixiter} and $Z_0(\eta,\xi) = 0$, we have
\begin{align}
&\limsup\limits_{k\rightarrow\infty} \sqrt[k]{\|Z_{k}(\eta,\xi)-X\|}
\nonumber\\
& =\rho\left((I\otimes\Delta+\Gamma\otimes I)^{-1}\left[I \otimes (E_2E_1^\top + X\overline{C})+(Q_1E_1^\top+\overline{C}X)\otimes I    \right]\right),
\end{align}
where $\otimes$ denotes the Kronecker product
(see \cite[Theorem~3.2]{GuoLaub2000}).
Second, let $M_X = I\otimes (\overline{A} - X\overline{C}) + (\overline{D} - \overline{C}X)^\top\otimes I$. Then, $M_X$ is a Z-matrix since both $\overline{A} - X\overline{C}$ and $\overline{D} - \overline{C}X$ are Z-matrices.
(see \cite[Remark~1.1]{GuoLaub2000}). Also, $M_X$ is a nonsingular matrix since any eigenvalue of $M_X$ is the sum of an eigenvalue of
$\overline{A} - X\overline{C}$ and
$\overline{D} - \overline{C}X$.  This implies that $M_X$ is a nonsingular M-matrix.
Third, if $M_X$ is a nonsingular M-matrix, then
\begin{equation}
\rho\left((I\otimes\Delta+\Gamma\otimes I)^{-1}\left[I \otimes (E_2E_1^\top + X\overline{C})+(Q_1E_1^\top+\overline{C}X)\otimes I    \right]\right) < 1,
\end{equation}
that is, $\limsup\limits_{k\rightarrow\infty} \sqrt[k]{\|Z_{k}(\eta,\xi)-X\|} < 1$.
(see \cite[Theorem~3.3]{GuoLaub2000})

\end{proof}


\section{Numerical Implementation and Comparisons}
To illustrate the consequence of the previous sections, numerical experiments, consisting of SDA and SI methods after the shifting technique, are presented to demonstrate our conclusion. All computations are performed in MATLAB/version~2010b on a iMac with an 2.8GHZ Intel Core i5 processor and 16GB main memory, using IEEE double-precision.

In the next implementations, the relative error for the SDA is defined by
\begin{align*}
\mbox{Err}_{SDA}=\max\left\{\frac{\|G_k-G_{k-1}\|_{\infty}}{\|G_{k}\|_{\infty}},\frac{\|H_k-H_{k-1}\|_{\infty}}{\|H_{k}\|_{\infty}}\right\},
\end{align*}
the relative error for  the SI  with no shift is defined by
\begin{align*}
\mbox{Err}_{SI}=\max\left\{\frac{\|m^{(k)}-m^{(k-1)}\|_{\infty}}{\|m^{(k)}\|_{\infty}},\frac{\|n^{(k)}-n^{(k-1)}\|_{\infty}}{\|n^{(k)}\|_{\infty}}\right\},
\end{align*}
the relative error for  the SI  with the shifting procedure is defined by
\begin{align*}
\mbox{Err}_{SIS}=\max\left\{\frac{\|M_k-M_{k-1}\|_{\infty}}{\|M_{k}\|_{\infty}},\frac{\|N_k-N_{k-1}\|_{\infty}}{\|N_{k}\|_{\infty}}\right\},
\end{align*}
and the relative normalized residual is defined by
\begin{align*}
\mbox{Res}=\frac{\|X_k\Gamma+\Delta X_k-(X_kq+e)(q^\top X_k+e^\top)\|_{\infty}}{\|X_k\|_{\infty}\|\Gamma\|_{\infty}+\|X_k\|_{\infty}\|\Delta\|_{\infty}+(\|X_k\|_{\infty}\|q\|_{\infty}+\|e\|_{\infty})(\|q^\top\|_{\infty} \|X_k\|_{\infty}+\|e^\top\|_{\infty})},
\end{align*}
where $X_k=G_k$ for the SDA algorithm, $X_k=T\circ(m^{(k)} {n^{(k)}}^\top)$ for the SI algorithm with no shift
and $X_k=T\circ(M_k N_k^\top)$ for the SI algorithm with shift.
All iteration methods are terminated whenever the relative errors or the relative normalized residual residuals are less than $n^2\epsilon$, where $\epsilon=2^{-52}\cong 2.22\cdot 10^{-16}$ be the machine zero.
\begin{example}\label{ex1}
In this example, we compare the methods for  finding the minimal nonnegative solution of \eqref{NARE} by using the shifting technique. We explain the efficiency of the SDA and SI applied to the shifted equations~\eqref{singalshift} and \eqref{barsubmat}, respectively. We consider~\eqref{NARE}  with $(\alpha, c)= (0,1)$. As suggested in~\cite{GuoLaub2000,LuNLAA05}, the constants  $c_i$ and $\omega_i$ are the nodes and weights, which are obtained by dividing the interval $[0, 1]$ into $n/4$ subinterval of equal length and applying to each subinterval the $4$-node Gauss--Legendre quadrature.

In table~\ref{5.1}, we report a comparison of residuals and the number of iterations
for the SDA with no shift, the SDA with a single shift, the SDA with double shifts
, the SI with no shift, the SI with a single shift, and the SI with double shifts and with size $n=32,64,128$, and $256$. From table~\ref{5.1}, we have the following two conclusions.

First, in the critical case $(\alpha,c) = (0,1)$,
 it is known that the SDA algorithm converges linearly. After applied to the shifted equation, the SDA algorithm converges quadratically.
As shown in Table~\ref{5.1}, the number of steps  required in the SDA algorithm with a single shift or double shifts  are
 around half of those of the SDA algorithm with no shift. Also, the computed solution of the shifted equations is more accurate than the one obtained with no shift. The numerical phenomena are in accordance with the theoretical discussion given in~\cite{Guo2007}.

Second, we randomly choose $\eta$ and $\xi$ from the set $\Omega$. Indeed, in Table~\ref{5.1}, we have $(\eta, \xi) =(\frac{1}{2\omega_1}, 0)$ for the single-shift problems and $(\eta,\xi) =(\frac{1}{2\omega_1}, \frac{-1}{2\omega_1})$ for the double-shift problems. We see that even with $10000$ steps, the solution obtained from the nonshifted problems can only have accuracy up to $10^{-8}$. On the other hand, the solution for the shifted problems can have the
accuracy better than $10^{-10}$ and a dramatical decrease in the number of steps required in the computation. Also, the iteration counts listed in Table~\ref{5.1} are in accord with Theorem~\ref{SITHM}.

\end{example}

\begin{table}[h!!!]\label{5.1}
\begin{center}
\begin{tabular}{ccccccc} \\
\hline
$n$ & SDA(no shift) & SDA(single shift)& SDA(double shifts)\\
\hline
32 & 9.7e-14(27)& 4.5e-15(11)& 7.4e-15(11)\\
64 & 4.2e-13(27)& 1.6e-14(12)& 1.9e-14(12)\\
128 & 1.7e-12(27)& 4.2e-14(13)& 6.1e-14(13)\\
256 & 6.8e-12(27)& 1.2e-13(14)& 1.4e-13(14)\\
\hline
\hline
$n$ &SI(no shift)&SI(single shift)&SI(double shifts) \\
\hline
32& * ($>$10000)&2.4e-13(164) &2.9e-13(40)\\
64& *($>$10000)&1.0-12(154) &1.3e-12(38)\\
128& *($>$10000)&4.0-12(145) &5.4e-12(36)\\
256& *($>$10000)&1.6-11(136) &2.2e-11(34)\\
\hline
\end{tabular}
\caption{Comparison of the residuals (and in parentheses the number of iterations) of the SDA and SI techniques.}
\end{center}
\end{table}

%

\section{Conclusion}
The challenging issues of applying the SDA algorithm to the shifted NARE problems are to
develop a well-defined sequence, to guarantee the convergence of the sequence, and to associate the solutions of the shifted problems with the original one. All these issues related to the structued NARE~\eqref{NARE} have been studied in our work. Numerical experiments show the improvement of the speed and accuracy while applying the SDA algorithm
to the shifted problems. Note that the bottleneck for applying this algorithm is to  compute the inverses of $(I_{n} - H_kG_k)$ and $(I_{n} - G_kH_k)$, which apparently have an $O(n^3)$ complexity. Compare with the Newton method, which has been shown to have $O(n^2)$ complexity~\cite{Bini2008}, an interesting problem worthy of further investigation is to reduce the computational cost by taking the specific structure of~\eqref{NARE} into account.

On the other hand, while applying the SI algorithm to the critical case, its convergence is very slow and has almost stopped. Through the shifting technology, a new iteration method has been introduced and preserve the linear convergence. Numerical experiments show that while considering the SI algorithm, the convergence with double shifts is much faster than the convergence with a single shift or no shift.
We believe the results we obtain are new in the field and could provide  considerable insight into the NARE problems.

%

\end{document}